\documentclass[12pt]{amsart}
\usepackage[utf8]{inputenc}

\usepackage{amsthm,amssymb}
\usepackage{amsmath}
\usepackage[framemethod=tikz]{mdframed}

\usepackage{url}
\usepackage[colorlinks,linkcolor=blue,anchorcolor=blue,citecolor=blue,backref=page]{hyperref}
\usepackage{color}
\usepackage{graphics,epsfig}
\usepackage{graphicx}
\usepackage{float} 
\usepackage[english]{babel}
\usepackage{mathtools}
\usepackage{todonotes}
\usepackage{url}
\usepackage[colorlinks,linkcolor=blue,anchorcolor=blue,citecolor=blue,backref=page]{hyperref}
\usepackage{breqn}
\usepackage{enumitem}
\usepackage{comment}
\usepackage{longtable}
\usepackage{bm}
\usepackage[nocompress]{cite}

\usepackage[norefs,nocites]{refcheck}

\usepackage{mathrsfs}
\hypersetup{breaklinks=true}

\newtheorem{thm}{Theorem}
\newtheorem{lem}[thm]{Lemma}

\newtheorem{cor}[thm]{Corollary}

\numberwithin{equation}{section}
\numberwithin{thm}{section}
\numberwithin{table}{section}

\def\squareforqed{\hbox{\rlap{$\sqcap$}$\sqcup$}}
\def\qed{\ifmmode\squareforqed\else{\unskip\nobreak\hfil
\penalty50\hskip1em\null\nobreak\hfil\squareforqed
\parfillskip=0pt\finalhyphendemerits=0\endgraf}\fi}

\def \balpha{\bm{\alpha}}
\def \bbeta{\bm{\beta}}

\def\cM{{\mathcal M}}

\def\cP{{\mathcal P}}

\def\cW{{\mathcal W}}

\def\inv{\mathrm{inv}} 
\def\sqrt{\mathrm{rt}}

\def \F {{\mathbb F}}

\def \Ki {\F_q(T)_\infty}

\def\\{\cr}
\def\({\left(}
\def\){\right)}

\newcommand{\ov}{\overbar}
\newcommand{\overbar}[1]{\mkern 1.5mu\overline{\mkern-1.5mu#1\mkern-1.5mu}\mkern 1.5mu}

 \newcommand{\Mod}[1]{\ (\mathrm{mod}\ #1)}

\makeatletter
\def\l@subsection{\@tocline{2}{0pt}{2.8pc}{1pc}{}}
\def\l@subsubsection{\@tocline{2}{0pt}{5pc}{7.5pc}{}}
\makeatother

\title[Bilinear Kloosterman sums in function fields ]{Bilinear Kloosterman sums in function fields and the distribution of irreducible polynomials}

 \author[C.~Bagshaw]{Christian Bagshaw}
 \address{School of Mathematics and Statistics, University of New South Wales.
 Sydney}
 \email{c.bagshaw@unsw.edu.au}


\begin{document}
\maketitle

\begin{abstract}
Inspired by the work of Bourgain and Garaev (2013), we provide new bounds for certain weighted bilinear Kloosterman sums in polynomial rings over a finite field. As an application, we build upon and extend some results of Sawin and Shusterman (2022). These results include bounds for exponential sums weighted by the M{\"o}bius function and a level of distribution for irreducible polynomials beyond 1/2, with arbitrary composite modulus. Additionally, we can do better when averaging over the modulus, to give an analogue of the Bombieri-Vinogradov Theorem with a level of distribution even further beyond 1/2. 
\end{abstract}

\tableofcontents

 \section{Introduction}

\subsection{Background}
Motivated by a range of applications, in recent years there has been notable effort dedicated to studying certain bilinear forms of Kloosterman sums. One important example are those of the form
\begin{align}\label{eq:bilinear_integercase}
    \sum_{\substack{1 \leq x_1 < N_1 \\ (x_1,m) = 1}}\sum_{\substack{1 \leq x_2 < N_2 \\ (x_2, m) = 1}}\alpha_{x_1}\beta_{x_2}e_m(a\ov{x_1}\ov{x_2})
\end{align}
for $a,m \in \mathbb{Z}$ and complex weighs $\balpha$ and $\bbeta$, where $e_m(x) = \exp(2\pi ix/m)$ and where $\ov{x}$ denotes the inverse of $x$ modulo $m$. Perhaps the most well-known application of bounds for \eqref{eq:bilinear_integercase} has been to estimate exponential sums over primes
\begin{align}\label{eq:vonmangoldt_integer}
    \sum_{\substack{1 \leq x < N \\ (x,m) = 1}}\Lambda(x)e_m(a\ov{x})
\end{align}
as in \cite{baker2012, fouvryshparlinski2011, BourgainGaraev2013, garaev2010, fouvrymichel1998, irving2014, korolevchange2020}, where $\Lambda$ denotes the Von Mangoldt function over $\mathbb{Z}$ (although later, by abuse of notation, this will denote the Von Mangoldt function in a different setting). Bounds for \eqref{eq:bilinear_integercase} have also found applications to the Brun-Titchmarsh theorem \cite{friedlanderiwaniec, BourgainGaraev2013, BourgainGaraev2014} and the distribution of fractional parts of fractions with modular inverses \cite{karatsuba1995fractional}. Higher dimensional analogues were also considered in \cite{luo1999, shparlinski2007}.

In their recent groundbreaking work \cite{sawinshusterman}, Sawin and Shusterman consider  analogues of \eqref{eq:bilinear_integercase} and \eqref{eq:vonmangoldt_integer} in polynomial rings over finite fields. They establish highly non-trivial bounds and apply them to obtain a level of distribution beyond 1/2 for irreducible polynomials to square-free modulus (for details see the discussion in Section \ref{sec:apps_mobius}), and even further to establish a strong and explicit form of the twin prime conjecture in that setting. Here we also consider \eqref{eq:bilinear_integercase} and \eqref{eq:vonmangoldt_integer} in function fields, but focus on working with arbitrary composite modulus. This includes improving some bounds from \cite{sawinshusterman} on sums of the form \eqref{eq:bilinear_integercase} and \eqref{eq:vonmangoldt_integer}, and also extending results regarding the level of distribution of irreducible polynomials, from square-free to arbitrary modulus. Furthermore, we establish a function field version of the Bombieri-Vinogradov Theorem with a level of distribution even further beyond $1/2$.

\subsection{General notation}

We fix an odd prime power $q = p^\ell$ and let $\F_q$ denote the finite field of order $q$.  Let $\F_q[T]$ denote the ring of univariate polynomials with coefficients from $\F_q$. Throughout, $F \in \F_q[T]$ will always denote an arbitrary polynomial of degree $r$. 

Next, we denote by $\Ki$ the field of Laurent series in $1/T$ over $\F_q$. That is,
$$\Ki = \left\{\sum_{i=-\infty}^na_iT^i ~:~n \in \mathbb{Z}, ~a_i \in \F_q, ~a_n \neq 0 \right\}.$$
We note that of course $\F_q[T] \subseteq \Ki$. On $\Ki$ we have a non-trivial additive character 
$$e\left( \sum_{i=-\infty}^na_iT^i\right) = \exp\left(\frac{2\pi i}{p}\text{Tr}(a_{-1}) \right)$$
where $\text{Tr}: \F_q \to \F_p$ is the absolute trace. Further, for any $F \in \F_q[T]$ we note that 
$$e_F(x) = e(x/F)$$
defines a non-trivial additive character of $\F_q[T]/\langle F(T) \rangle$. See \cite{Hayes1966} for additional details. 

We will let $\cM$ and $\cP$ be the set of all monic and all monic irreducible polynomials, respectively.  For a positive integer $n$ we will let $\cM_n$ be the set of monic polynomials of degree $n$. 

We can also define an analogue of the M{\"o}bius function in $\F_q[T]$, as
\begin{align*}
    \mu(x)
    =
    \begin{cases}
        0, &x\text{ is not square-free}\\
        (-1)^k, &x = hP_1^{e_1}\cdots P_k^{e_k} \text{ for $h \in \F_q$ and for distinct $P_i \in \cP$.}
    \end{cases}
\end{align*}
Similarly, we can define the von Mangoldt function
\begin{align*}
    \Lambda(x)
    =
    \begin{cases}
        \deg P, &x = hP^k\text{ for some $P \in \cP$ and $h \in \F_q$}\\
        0, &\text{otherwise}.
    \end{cases}
\end{align*}

Finally, given some $x \in \F_q[T]$, $\ov{x}$ will denote the inverse of $x$ modulo $F$ (unless it is specified that the inverse should be taken to a different modulus). Also, $\epsilon$ will denote some small constant (unless otherwise specified).

\section{Results}\label{sec:results}
\subsection{Bilinear Kloosterman sums}\label{sec:results_bilinear}
Given positive integers $m$ and $n$, sequences of complex weights 
\begin{align}\label{eq:weights}
    \balpha = (\alpha_{x_1})_{\deg x_1 < m}, ~\bbeta = (\bbeta_{x_2})_{\deg x_2 < n} \text{ with }\|\balpha\|_\infty, \|\bbeta\|_\infty < q^{o(r)}
\end{align}
and $a \in \F_q[T]$ we define the bilinear Kloosterman sum
$$\cW_{F,a}(m,n;\balpha, \bbeta) = \sum_{\substack{\deg x_1 < m \\ (x_1, F) = 1}}\sum_{\substack{\deg x_2 < n \\ (x_2, F) = 1}}\alpha_{x_1}\beta_{x_2}e_F(a\ov{x_1x_2}).$$
We will be interested in improving upon the trivial bound $q^{m+n + \epsilon r}$. 
As mentioned previously, bounds on sums of this form are used as tools to establish some of the main results in \cite{sawinshusterman}. Here we take a different approach to bounding these sums which can hold for arbitrary $F$, based on the ideas of Bourgain and Garaev \cite{BourgainGaraev2013}, Garaev \cite{garaev2010}, Fouvry and Shparlinski \cite{fouvryshparlinski2011}, Banks, Harcharras and Shparlinski \cite{BanksHarcharrasShparlinski2003} and Irving \cite{irving2014}. 

More flexible bounds, given explicitly in terms of additive energies of modular inversions, are stated in Section \ref{sec:bilinearsums}. These would imply function field analogues of most of the bounds in \cite{BourgainGaraev2013}. But the following will be the most useful for our purposes.

\begin{thm}\label{thm:bilinear_savings}
Let $\epsilon > 0$, and let $a,F \in \F_q[T]$ be coprime with $\deg F = r$. Then for any positive integers $n$ and $m$ satisfying 
$$n \geq r\epsilon \text{ and } m \geq r(1/4 + \epsilon)$$ 
and weights as in \eqref{eq:weights}, we have 
\begin{align*}
    \cW_{F,a}(m,n;\balpha, \bbeta) \ll_{\epsilon}  q^{m+n - r\delta}
 \end{align*}
 for some $\delta= \delta(\epsilon) > 0$. 
\end{thm}

\subsection{Kloosterman sums with the M{\"o}bius function}\label{sec:apps_mobius}
As in \cite{SawinShusterman2, sawinshusterman}, we next consider sums of the form 
\begin{align}\label{eq:functionfield_mobiussum}
    \sum_{\substack{\deg x < n \\ (x,F) = 1}}\mu(x)e_F(a\ov{x}),
\end{align}
and seek improvement over the trivial bound $q^n$, for $n$ as small as possible in comparison to $r$. We note that we are working with the M{\"o}bius function as opposed to the Von Mangoldt function as in \eqref{eq:vonmangoldt_integer}, but the similarity between Vaughan's identity \cite[Propositions 13.4 and 13.5]{IwaniecKowalski2004} for $\mu$ and $\Lambda$ allows for both of these to be treated very similarly. 

Analogous results dealing with sums as in \eqref{eq:functionfield_mobiussum} over the integers always require $\gcd(a,F) = 1$. But because the analogue of the Generalized Riemann Hypothesis (GRH) holds in $\F_q[T]$, we can drop this condition (with some additional analytic effort). 

A special case of \cite[Theorem 1.13]{SawinShusterman2} is the following: let $\epsilon > 0$ and suppose $F$ is irreducible. If 
\begin{align}\label{eq:mobius_q_bound}
    q > 4e^2\left(1 + \frac{3}{2p} \right)^{p/\epsilon}p^2
\end{align}
then for $n > r\epsilon$ we have 
$$\sum_{\substack{\deg x < n \\ (x,F) = 1}}\mu(x)e_F(a\ov{x}) \ll_{\epsilon} q^{n(1-\delta)}$$
for some $\delta = \delta(\epsilon) > 0$. In summary, this implies that for any $\epsilon > 0$, one obtains a power savings over the trivial bound for any $n > r\epsilon$ (for sufficiently large $q$ in terms of $p$ and $\epsilon$). This achievement of a power savings in arbitrarily small intervals far surpasses any previous work in this area. 

Here we consider what can be said without these restrictions on $q$, and for arbitrary composite modulus $F$. Using Theorem \ref{thm:bilinear_savings} together with classical ideas regarding Vaughan's identity, we show the following.
\begin{thm}\label{thm:mobius_r/2}
Let $\epsilon > 0$ and $a,F \in \F_q[T]$ with $\deg F = r$. For any positive integer $n$ satisfying $n > r(1/2 + \epsilon)$, 
$$\sum_{\substack{\deg x < n \\ (x,F) = 1}}\mu(x)e_F(a\ov{x}) \ll_\epsilon q^{n(1-\delta)}$$
for some $\delta = \delta(\epsilon) > 0$.
\end{thm}

For a comparison with \cite[Theorem 1.13]{SawinShusterman2}, the most important point is that this result holds for arbitrary modulus $F$ as opposed to only irreducible modulus. But this also does not require the restriction on $q$ as in \eqref{eq:mobius_q_bound}. Thus, Theorem \ref{thm:mobius_r/2} gives an improvement for irreducible modulus when $q = p^\ell$ and
\begin{align*}
    &p \in \{3\} \text{ and } \ell < 8, \\
    &p \in \{5,7\} \text{ and } \ell < 6, \\
    &p \in \{11,...,23\} \text{ and } \ell < 5, \\
    &p \in \{29,...,587\} \text{ and } \ell < 4, \\
    &p \in \{593,...\} \text{ and } \ell < 3. \\
\end{align*}
Another important avenue to pursue with regard to these sums is obtaining more explicit (and larger) savings over the trivial bound. In \cite{sawinshusterman}, these are required for applications. For square-free modulus $F$, \cite[Theorem 1.7]{sawinshusterman} demonstrates
\begin{align}\label{eq:SS_1/32}
    \sum_{\substack{\deg x < n \\ (x,F) = 1}}\mu(x)e_F(a\ov{x}) \ll_{\epsilon} q^{3r/16 + 25n/32 + \epsilon n},
\end{align}
which is non-trivial when $n > 6r/7.$

This can be improved, and again can be extended to arbitrary modulus. The proof makes use of some ideas of Garaev \cite{garaev2010} and Fouvry and Shparlinski \cite{fouvryshparlinski2011}, but we can do better in $\F_q[T]$ because of GRH. 
\begin{thm}\label{thm:mobius_3/4}
Let $a,F \in \F_q[T]$ with $\deg F = r$ and let $n$ denote a positive integer. Then for any $\epsilon > 0$, 
$$\sum_{\substack{\deg x < n \\ (x,F) = 1}}\mu(x)e_F(a\ov{x}) \ll_{\epsilon} q^{15n/16 + \epsilon n} + q^{2n/3 + r/4 + \epsilon n}.$$
\end{thm}
This is non-trivial when $n > 3r/4$,  and gives a savings of $q^{r/16}$ over the trivial bound when $n \approx r$. Also, this always improves on \eqref{eq:SS_1/32}. 

For applications, we will also make use of the following variant of a result of Irving \cite{irving2014} which gives an improvement on average over the modulus. 

\begin{thm}\label{thm:irving}
    For any positive integers $n$ and $r$ and any $\epsilon > 0$
    \begin{align*}
        \sum_{\deg F = r}\max_{a \in \F_q[T]}\bigg{|}\sum_{\substack{\deg x < n \\ (x,F) = 1}}&\mu(x)e_F(a\ov{x})\bigg{|} \\
        &\ll_{\epsilon} q^{r + n\epsilon}(q^{9n/10 } + q^{r/6 + 13n/18} + q^{13n/8 - 5r/6}).
    \end{align*}
\end{thm}
This is again non-trivial when $n > 3r/4$, with a savings of $q^{r/10}$ over the trivial bound when $n \approx r$.

\subsection{Level of distribution of irreducible polynomials}\label{sec:apps_irreds}

The main application in \cite{sawinshusterman} of the sums considered in the previous section is to obtain a level of distribution beyond $1/2$ for irreducible polynomials in arithmetic progressions. In particular, that means non-trivial bounds for
\begin{align*}
    \Bigg{|}\sum_{\substack{x \in \cM_n \\ x \equiv a \Mod{F}}}\Lambda(x) - \frac{q^n}{\phi(F)}\Bigg{|}
\end{align*}
when $n > r/2$. The start of \cite[section 17.1]{IwaniecKowalski2004} gives a good background on this problem over $\mathbb{Z}$, but in summary it is a classical problem in number theory to show 
\begin{align}\label{eq:vonmangoldt_asymptotic}
    \sum_{\substack{x \in \cM_n \\ x \equiv a \Mod{F}}}\Lambda(x) \sim \frac{q^n}{\phi(F)} \text{ uniformly for } r \leq n \omega
\end{align}
for $\omega < 1$ as large as possible. The strongest analogous results over $\mathbb{Z}$ only imply that under the assumption of GRH, \eqref{eq:vonmangoldt_asymptotic} holds for $\omega < 1/2$, although it is conjectured that this should hold for any $\omega < 1$; again, see \cite[section 17.1]{IwaniecKowalski2004}. 

In $\F_q[T]$, Sawin and Shusterman \cite[Theorem 1.9]{sawinshusterman} move beyond this barrier of $1/2$ for square-free modulus $F$  by showing (for sufficiently large but fixed $q$ in terms of $\omega$ and $p$) that 
\begin{align}\label{eq:ss_og_vonmangoldt}
    \eqref{eq:vonmangoldt_asymptotic} \text{ holds for any } \omega < 1/2 + 1/126 \text{ and square-free } F.
\end{align}

Sawin subsequently gives another ground-breaking improvement in \cite[Theorem 1.2]{Sawin2023} to achieve the conjectured value of $\omega$ for square-free modulus, by showing (for sufficiently large but fixed $q$ in terms of only $\omega$) that 
\begin{align}\label{eq:s_top_vonmonagoldt}
   \eqref{eq:vonmangoldt_asymptotic} \text{ holds for any } \omega < 1 \text{ and square-free } F.
\end{align}

Again one may ask whether we can move past the barrier of $\omega < 1/2$ for arbitrary modulus. The methods used to show \eqref{eq:s_top_vonmonagoldt} are very specialized to square-free modulus, and it is probably infeasible to make these work more generally. But, by inserting our Theorem \ref{thm:mobius_3/4} into the proof of \eqref{eq:ss_og_vonmangoldt}, we have the following. 
\begin{thm}\label{cor:vonmangoldt}
Fix $\omega < 1/2 + 1/62$, and suppose
$$q > 
p^2e^2\left(\frac{16-\omega}{16-31\omega} \right)^2.$$
Then for any coprime $a,F \in \F_q[T]$ with $\deg F = r$, and any positive integer $n$ satisfying $r \leq \omega n$ we have 
$$\sum_{\substack{x \in \cM_n \\ f \equiv a \Mod{F}}}\Lambda(x) - \frac{q^n}{\phi(F)} \ll_{\omega} q^{n-r(1+\delta)}$$
for some $\delta = \delta(\omega) > 0$. 
\end{thm}

While the holds for arbitrary modulus $F$, we do note that for square-free modulus Sawin's result \cite[Theorem 1.2]{Sawin2023} always gives a more relaxed condition on $q$. 

We can also use Theorem \ref{thm:irving} to do better on average. That is, when considering an analogue of the Bombieri-Vinogradov Theorem.

\begin{thm}\label{thm:bombierivinogradov}
    Fix $\omega < 1/2 + 1/38$, and suppose 
    $$q > p^2e^2\left(\frac{10-\omega}{10-19\omega} \right)^2.$$
    Then for any positive integers $R$ and $n$ satisfying $R \leq \omega n$ we have 
    \begin{align*}
            \sum_{\deg F < R}\max_{(a,F) = 1}\Bigg{|}\sum_{\substack{x \in \cM_n \\ x \equiv a \Mod{F}}}\Lambda(x) - \frac{q^n}{\phi(F)}\Bigg{|} \ll_{\omega} q^{n-R\delta}
    \end{align*}
for some $\delta = \delta(\omega) > 0$. 
\end{thm}

 \section{Preliminaries}
 Throughout this section, $F$ always denotes an arbitrary polynomial of degree $r$, and $\epsilon > 0$ is always some small positive constant.

As a general preliminary, we will repeatedly make use of the following from \cite[Lemma 1]{CS2013}. 

\begin{lem}\label{lem:divisors}
    The number of divisors of any $x \in \F_q[T]$ is $O_{\epsilon}(q^{\epsilon\deg x})$.
\end{lem}
 \subsection{Sums involving the M{\"o}bius function}
We will need a number of results regarding cancellations in sums of the M{\"o}bius function. First, we recall the following elementary result from \cite[Chapter 2, Ex 12]{Rosen2013}. 
\begin{lem}\label{lem:Rosen}
    For any positive integer $n$, 
    $$\sum_{\substack{x \in \cM_n}}\mu(x) = 
    \begin{cases}
        -q, &n=1,\\
        0, &n > 1.
    \end{cases}$$
\end{lem}

The next result is found in \cite[Theorem 2]{BhowmickLeLiu2017}. We observe that there is a mistake in the statement of this result in \cite{BhowmickLeLiu2017}, but it is correct as stated here (see the discussion in Section 4.5 of \cite{Bagshaw2023_PAMS}).  
\begin{lem}\label{lem:BLL}
    Suppose 
    $$r \geq 10^4 \text{ and }\frac{\log r}{\log \log r} \geq \log q.$$
    Let $\chi$ denote a non-principal character modulo $F$. Then for any positive integer $n$
    $$\Big{|}\sum_{\substack{x \in \cM_n}}\mu(x)\chi(x) \Big{|} \leq q^{\frac{n}{2} + \frac{n\log\log r}{\log r} + 8q\frac{r}{\log^2r}\log_qe}.$$
\end{lem}
We will also make use of the following from \cite{H2020}.
\begin{lem}\label{lem:Han}
    Let $\chi$ denote a non-principal character modulo $F$. Then for any positive integer $n$
    $$\Big{|}\sum_{\substack{x \in \cM_n}}\mu(x)\chi(x)\Big{|}  \leq q^{n/2}\binom{n+r-2}{n}.$$
\end{lem}
The previous two results can be combined and simplified for our purposes. This is classical in the literature, but we include brief details for completeness.
\begin{cor}\label{cor:Han+BLL}
    For any positive integer $n \geq r$ and any non-principal character $\chi$ modulo $F$ we have 
    $$\sum_{\substack{\deg x < n}}\mu(x)\chi(x) \ll_{\epsilon} q^{n(1/2 + \epsilon)}.$$
\end{cor}
\begin{proof}
    Let $S$ denote the sum in question. We split our sum into intervals depending on the degree of $x$ and write  
    $$S 
    \ll \sum_{i=0}^{n-1}\Bigg{|}\sum_{\substack{x \in \cM_n}}\mu(x)\chi(x)\Bigg{|}.  $$
    This implies there exists some integer $t < n$ such that 
    $$S \ll n\Bigg{|}\sum_{\substack{x \in \cM_t}}\mu(x)\chi(x)\Bigg{|}.$$
    First, if $t < n/2$ then the result follows trivially. So suppose $n/2 \leq t \leq n$. 
    If $r < \log n$ then by Lemma \ref{lem:Han}, 
    \begin{align*}
        S \ll nq^{t/2}\binom{t+\log n-2}{t} \ll_{\epsilon} q^{n(1/2 + \epsilon)}.
    \end{align*}
    Finally, if $r \geq \log n$ then since $t \geq n/2 \geq r/2$, Lemma \ref{lem:BLL} implies
    \begin{align*}
        S 
        \ll nq^{t(\frac{1}{2} + \frac{\log\log r}{\log r} + 8q\frac{r}{t\log^2r}\log_qe)}
        &\ll nq^{t(\frac{1}{2} + \frac{\log\log r}{\log r} + 16q\frac{1}{\log^2r}\log_qe)}\\
        &\ll_{\epsilon} q^{n(1/2 + \epsilon)}.
    \end{align*}
\end{proof}

This now implies the following, which is again well-known but we include details for completeness.
\begin{cor}\label{cor:mobius_arith}
    Let $a \in \F_q[T]$ with $\gcd(a,F) = 1$. Then for any positive integer $n$, 
    $$\sum_{\substack{\deg x < n \\ x \equiv a \Mod{F}}}\mu(x) \ll_{\epsilon} q^{n(1/2 + \epsilon)}.$$
\end{cor}
\begin{proof}
    Of course, if $n < r$ then this is trivial, so we assume otherwise. Using the orthogonality of multiplicative characters we may write
    \begin{align*}
        \sum_{\substack{\deg x < n \\ x \equiv a \Mod{F}}}\mu(x) 
        &= \frac{1}{\phi(F)}\sum_{\chi \Mod{F}}\ov{\chi(a)}\sum_{\deg x < n}\mu(x)\chi(x).
    \end{align*}
    The trivial character contributes only $O(1)$ by Lemma \ref{lem:Rosen}. To bound the rest, we can apply the triangle inequality and then Corollary \ref{cor:Han+BLL} to reach the desired result. 
\end{proof}

The following is a special case of \cite[Theorem 4.5]{sawinshusterman}, which significantly improves upon the previous result when $r$ is close to $n$ (with some restrictions on the size of $q$).

\begin{lem}\label{lem:sawinshusterman_mobius_dist}
    Let $\epsilon> 0$ and $0 < \beta < 1/2$, and suppose
    $$q > \left(\frac{\epsilon + 2}{\epsilon}~{pe}\right)^{\frac{2}{1-2\beta}}.$$
    Then for any non-negative integer $n \geq (1+\epsilon)r$ and any $a \in \F_q[T]$ coprime to $F$ we have
    $$\sum_{\substack{x \in \cM_n \\ x \equiv a \Mod{F} }}\mu(x) \ll_{\epsilon, \beta} q^{(n-r)(1-\beta/p)}. $$
\end{lem}

Finally, the next result is \cite[Proposition 5.2]{sawinshusterman}. Originally this was only stated for square-free $F$, but in fact that immediately implies that it holds for arbitrary $F$.

\begin{lem}\label{lem:sawinshusterman_mainterm}
    For any positive integer $d$, 
    $$\sum_{k=1}^dkq^{-k}\sum_{\substack{x \in \cM_k \\ (x,F) = 1}}\mu(x) = -\frac{q^r}{\phi(F)} + q^{o(r+d)-d}.$$
\end{lem}

 \subsection{The Weil bound for Kloosterman sums}
To effectively bound the bilinear Kloosterman sums introduced in Section \ref{sec:results_bilinear}, we will need a few well-known estimates regarding complete and incomplete Kloosterman sums. First, we need the following orthogonality relation (see \cite[Corollary 4.2]{Bagshaw2023}). 
\begin{lem}\label{lem:orthogonality}
    For any $a \in \F_q[T]$ with $\deg a < r$ and positive integer $n$,
    $$\sum_{\deg x < n}e_F(ax) = 
    \begin{cases}
        q^n, &\deg a < r-n\\
        0, &\text{otherwise}.
    \end{cases}
    $$
\end{lem}
The following is from \cite[Lemma A.13]{Bagshaw2023}.
\begin{lem}\label{lem:weil}
    For any $a,b\in \F_q[T]$, 
    $$\Bigg{|}\sum_{\substack{\deg x < r \\ (x,F) = 1}}e_F(ax + b\ov{x} )\Bigg{|} \ll_{\epsilon} q^{r/2 + \deg(a,b,F)/2 + r\epsilon}.$$
\end{lem}

Next, Lemma \ref{lem:orthogonality} and Lemma \ref{lem:weil} implies the following. 
\begin{lem}\label{lem:weil_incomplete}
   For any $b \in \F_q[T]$ and positive integer $n \leq r$, 
    $$\Bigg{|}\sum_{\substack{\deg x < n \\ (x,F) = 1}}e_F(b\ov{x} )\Bigg{|} \ll_{\epsilon} q^{r/2 + \deg(b,F)/2 + \epsilon r}.$$
\end{lem}
\begin{proof}
By applying Lemma \ref{lem:orthogonality} and then rearranging and applying Lemma \ref{lem:weil}, 
    \begin{align*}
        \Bigg{|}\sum_{\substack{\deg x < n \\ (x,F) = 1}}e_F(b\ov{x} )\Bigg{|}
        &= q^{n-r}\Bigg{|}\sum_{\substack{\deg x < r \\ (x,F) = 1}}e_F(b\ov{x} )\sum_{\deg a < r-n}e_F(ax)\Bigg{|}\\
        &\ll_{\epsilon} q^{n-r}\sum_{\deg a < r-n}q^{r/2 + \deg(a,b,F)/2 + \epsilon r}\\
        &\ll_\epsilon q^{r/2 + \deg(b,F)/2 + \epsilon r}.
    \end{align*}
\end{proof}

We will also make use of the following. 
\begin{lem}\label{lem:weil_different_modulus}
    Let $b,u \in \F_q[T]$ and suppose $\deg u = O(r)$. Then  
    $$\Bigg{|}\sum_{\substack{\deg x < r \\ (x,uF) = 1}}e_F(b\ov{x} )\Bigg{|} \ll_\epsilon q^{r/2 + \deg(b,F)/2 + \epsilon r}.$$
\end{lem}
\begin{proof}
Without loss of generality we may suppose that $(u,F) = 1$. We recall the identity
    $$\sum_{\substack{d|x \\ \text{ d monic}}}\mu(d) = \begin{cases}
        1, &\deg x = 0, \\
        0, &\text{ otherwise}.
    \end{cases}$$
    Thus a typical application of inclusion-exclusion implies 
    \begin{align*}
        \sum_{\substack{\deg x < r \\ (x,uF) = 1}}e_F(b\ov{x} )
        &= \sum_{\substack{\deg x < r \\ (x,F) = 1}}e_F(b\ov{x} )\sum_{\substack{d|(u,x)\\ d\text{ monic}}}\mu(d)\\
        &= \sum_{\substack{d|u \\ d \text{ monic}}}\mu(d)\sum_{\substack{\deg x < r \\ (x,F) = 1 \\ d|x}}e_F(b\ov{x})\\
        &= \sum_{\substack{d|u \\ d \text{ monic}}}\mu(d)\sum_{\substack{\deg x < r - \deg d \\ (x,F) = 1}}e_F(b\ov{dx}).
    \end{align*}
    Now applying the triangle inequality and Lemmas \ref{lem:divisors} and \ref{lem:weil_incomplete} concludes the proof. 
\end{proof}
\subsection{Additive energy of modular inversions}
We will repeatedly make use of bounds regarding the number of solutions to certain equations with modular inverses. For positive integers $n$ and $k$ we define $I_{F,a,k}(n)$ to count the number of solutions to 
\begin{align}\label{eq:sum_of_inv}
  \ov{x_1} +\cdots+ \ov{x_k} \equiv a \Mod{F}, ~\deg x_i < n, 
\end{align}
and 
\begin{align*}
    E_{F,k}^\inv(n) = \sum_{a \Mod{F}}I_{F,a,k}(n)^2.
\end{align*}
This can be considered a measure of the additive energy of the set 
$$\{\ov{x} \Mod{F} : \deg x < n\}.$$
First, we will make use of the following from \cite{BagshawKerr2023}. 

\begin{lem}\label{lem:inverse_energy}
   Let $k$ be a fixed positive integer. Then for any positive integer $n \leq r$,
    $$E_{F,k}^\inv(n) \ll_{\epsilon, k} q^{kn + \epsilon n} + q^{n(3k-1)-r+\epsilon n}.$$
    In particular, this implies
    \begin{align*}
    E_{F,k}^{\inv}(n) \ll_{\epsilon, k} 
    \begin{cases}
        q^{kn + \epsilon n}, &n < r/(2k-1),\\
        q^{n(3k-1)-r + \epsilon n}, &r/(2k-1) \leq n \leq r/k,\\
        q^{n(2k-1) + \max\{0, n-r\}}, & r/k < n
    \end{cases}
\end{align*}
by using the trivial bound when $r/k < n$. 
\end{lem}

This can be improved upon when $k=2$, and the following is a generalisation of \cite[Theorem 2.5]{BS2022} to arbitrary modulus. 
\begin{lem}\label{lem:inverse_energy_k=2}
For any positive integer $n \leq r$, 
$$E_{F,2}^{\inv}(n) \ll_{\epsilon} q^{2n+\epsilon n}+q^{7n/2-r/2+\epsilon n}. $$
In particular this implies
\begin{align*}
    E_{F,2}^{\inv}(n) \ll_{\epsilon} 
    \begin{cases}
        q^{2n + \epsilon n}, &n < r/3,\\
        q^{7n/2 - r/2 + \epsilon n}, &r/3 \leq n \leq r,\\
        q^{4n - r}, & r < n
    \end{cases}
\end{align*}
by using the trivial bound when $r < n$. 
\end{lem}
\begin{proof}
Recall that we are counting the number of solutions to 
\begin{align}\label{eq:energy_k=2}
    \ov{x_1} + \ov{x_2} \equiv \ov{x_3} + \ov{x_4} \Mod{F}, ~~\deg x_i < n.
\end{align}
This is trivially satisfied if $x_1 \equiv -x_2 \Mod{F}$ and $x_3 \equiv -x_4 \Mod{F}$. Thus we can write 
$$E_F^{\inv}(n) = E_F^{\inv*}(n) + O(q^{2n})  $$
where $E_F^{\inv*}(n)$ counts the number of solutions to \eqref{eq:energy_k=2} where each side is non-zero. Next, we observe
$$E_F^{\inv*}(n) = \sum_{0 \leq \deg a < r}I_{F,a,2}(n)^2.$$
Using \cite[Lemma 5.3]{Bagshaw2023}, we have 
$$I_{F,a}(n) \ll_{\epsilon} q^{\epsilon n}(1 + q^{3n/2-r/2} + q^{2n+\deg(a,F)-r})$$
which implies 
\begin{align}\label{eq:E_inv_split_gcd}
    E_F^{\inv*}(n) 
    &\ll_{\epsilon} q^{\epsilon n}\sum_{0 \leq \deg a < r}I_{F,a}(n)(1 + q^{3n/2-r/2} + q^{2n+\deg(a,F)-r}) \nonumber\\
    &\ll_\epsilon q^{\epsilon n}(q^{2n}+ q^{7n/2-r/2}) + q^{2n-r + \epsilon n}\sum_{0 \leq \deg a < r}q^{\deg(a,F)}I_{F,a}(n).
\end{align}
To deal with the sum in this expression, we write
\begin{align}\label{eq:E_inv_split_gcd_2}
    \sum_{0 \leq \deg a < r}q^{(a,F)}I_{F,a}(n)
    &= \sum_{\substack{d|F \\ d\text{ monic}}}q^{\deg d}\sum_{\substack{0 \leq \deg a < r \\ (a,F) = d}}I_{F,a}(n).
\end{align}
First, if $\ov{x_1} + \ov{x_2} \equiv a \Mod{F}$ for some $(x_1, x_2)$ then of course $a$ is uniquely determined. Next, given some $a$ in the inner sum on the right of \eqref{eq:E_inv_split_gcd_2}, write $a = a_0d$ and $F = F_0d$ where $\gcd(a,F) = d$. Thus if 
$$\ov{x_1} + \ov{x_2} \equiv a_0d \Mod{F_0d}.$$
then
\begin{align}\label{eq:modd}
    x_1 + x_2 \equiv 0 \Mod{d}
\end{align}
implying
\begin{align*}
    &\sum_{0 \leq \deg a < r}q^{(a,F)}I_{F,a}(n)\\
    &\qquad\qquad\leq \sum_{\substack{d|F \\ d\text{ monic}}}q^{\deg d}\#\{(x_1, x_2) : \deg x_i < n, ~ x_1 + x_2 \equiv 0 \Mod{d}\}.
\end{align*}
If $\deg d \geq n $ then there are no solutions to \eqref{eq:modd} with $\deg x_i < n$ unless $x_1 = -x_2$, but we have already eliminated this case. If $\deg d < n$, then for any choice of $x_1$ there are at most $q^{n-\deg d}$ possibilities for $x_2$. Thus by equations \eqref{eq:E_inv_split_gcd} and \eqref{eq:E_inv_split_gcd_2} and Lemma \ref{lem:divisors} we can conclude 
\begin{align*}
    E_F^{\inv*}(n) 
    &\ll_\epsilon
q^{\epsilon n}(q^{2n}+ q^{7n/2-r/2}) + q^{2n-r}\sum_{\substack{d|F \\ d\text{ monic} \\ \deg d < n}}q^{\deg d}(q^{2n-\deg d})\\ 
&\ll_\epsilon q^{\epsilon n}(q^{2n}+ q^{7n/2-r/2} + q^{4n-r}).
\end{align*}
as desired. 
\end{proof}

Also, ideas from \cite{fouvryshparlinski2011} show that these can be improved when averaging over the modulus. 

\begin{lem}\label{lem:add_energy_average}
    Let $n,r$ and $k$ be positive integers. Then  
    $$\sum_{\deg F = r}E_{F,k}^{\inv}(n) \ll_{\epsilon, k} q^{r + n(k + \epsilon)} + q^{n(2k + \epsilon)}.$$
\end{lem}
\begin{proof}
    By clearing denominators, it suffices to count solutions to 
    $$\sum_{i=1}^k\prod_{\substack{j=1 \\ j \neq i}}^{2k}x_i  - \sum_{i=k+1}^{2k}\prod_{\substack{j=1 \\ j \neq i}}^{2k}x_i\equiv 0 \Mod{F}, ~\deg x_i < n, ~\deg F = r. $$
    If the left-hand side of the expression is equal to $0$, then \cite[Lemma 2.6]{SZ2018} implies there are at most $O_{\epsilon, k}(q^{r + nk + n\epsilon})$ solutions. Otherwise, we must have that $F$ divides the left-hand side, yielding at most $O_{\epsilon}(q^{n\epsilon})$ choices for $F$, implying at most $O_{\epsilon}(q^{2kn + n\epsilon})$ solutions in total. 
\end{proof}

\section{Bilinear Kloosterman sums}\label{sec:bilinearsums}

We can now present our results regarding bilinear Kloosterman sums. Before proving Theorem \ref{thm:bilinear_savings} we will present a few more general results. The following can give a power-savings over the trivial bound when used in conjunction with Lemma's \ref{lem:inverse_energy} and \ref{lem:inverse_energy_k=2}, although for flexibility we do not substitute these bounds yet. We note that the case $k_1=k_2=2$ recovers \cite[Theorem 2.5]{BS2022} when Lemma \ref{lem:inverse_energy_k=2} is applied, although this generalises it to composite modulus. 
\begin{lem}\label{lem:bilinear}
Let $\epsilon > 0$. Let $k_1$ and $k_2$ denote positive integers and $a,F \in \F_q[T]$ with $\gcd(a,F) = 1$ and $\deg F = r$. Then for any positive integers $n$ and $m$ and weights as in \eqref{eq:weights}, we have 
\begin{align*}
     &\cW_{F,a}(m,n;\balpha, \bbeta) \ll_{\epsilon} q^{m+n + \epsilon r}\bigg{(}E_{F,k_1}(m)E_{F,k_2}(n)q^{r-2nk_2-2mk_1}\bigg{)}^{\frac{1}{2k_1k_2}}.
 \end{align*}
\end{lem}
\begin{proof}
      Let $S = |\cW_F(m,n;\balpha, \bbeta)|$. Applying H{\"o}lders inequality yields 
\begin{align*}
    S^{k_2} 
    &\ll_{\epsilon} q^{\epsilon r k_2/2}~q^{m(k_2-1)}\sum_{\substack{\deg x_1 < m \\ (x_1, F) = 1} }\Bigg{|} \sum_{\substack{\deg x_2 < n \\ (x_2, F) = 1}}\beta_{x_2}e_F(a\ov{x}_1\ov{x}_2) \Bigg{|}^{k_2}.
\end{align*}
Expanding the inner sum and rearranging then yields
\begin{align*}
    S^{k_2} 
    &\ll_\epsilon q^{\epsilon r k_2/2}~q^{m(k_2-1)}\sum_{\substack{\deg x_1 < m  \\ (x_1, F) = 1}}\Bigg{|} \sum_{\substack{y_1,...,y_{k_2} \\ \deg y_i < n \\ (y_i,F)=1}}\beta_{y_1}...\beta_{y_{k_2}}e_F(a\ov{x}_1(\ov{y}_1 + ...+\ov{y}_{k_2}) )\Bigg{|}\\
    &= q^{\epsilon r k_2/2}~q^{m(k_2-1)}\sum_{\substack{\deg x_1 < m \\(x_1, F) = 1}}\gamma_{x_1}\sum_{\substack{y_1,...,y_{k_2} \\ \deg y_i < n \\ (y_i, F) = 1}}\beta_{y_1}...\beta_{y_{k_2}}e_F(a\ov{x}_1(\ov{y}_1 + ...+\ov{y}_{k_2}))\\
    &\ll q^{\epsilon r k_2}~q^{m(k_2-1)}\sum_{\substack{y_1,...,y_{k_2} \\ \deg y_i < n \\ (y_i, F) = 1}}\Bigg{|}\sum_{\substack{\deg x_1 < m \\ (x_1, F) = 1}}\gamma_{x_1}e_F(a\ov{x}_1(\ov{y}_1 + ...+\ov{y}_{k_2}) )\Bigg{|}
\end{align*}
for some $|\gamma_{x_1}| \leq 1$. By Applying H{\"o}lder's inequality again, we have 
\begin{align*}
    S^{k_1k_2} \ll_{\epsilon} q^{\epsilon r k_1k_2}~&q^{mk_1(k_2-1)+ nk_2(k_1-1)}\\
    &\times\sum_{\substack{y_1,...,y_{k_2} \\ \deg y_i < n \\ (y_i, F) = 1}}\Bigg{|} \sum_{\substack{\deg x_1 < m \\(x_1, F) = 1}}\gamma_{x_1}e_F(a\ov{x}_1(\ov{y}_1 + ...+\ov{y}_{k_2}) )\Bigg{|}^{k_1}. 
\end{align*}
This can be rewritten as 
\begin{align*}
    S^{k_1k_2} \ll_{\epsilon} ~&q^{\epsilon r k_1k_2 + mk_1(k_2-1)+ nk_2(k_1-1)} \\
    &\qquad\sum_{\substack{\deg \lambda < r}}I_{F, \lambda, k_2}(n)\Bigg{|}\sum_{\substack{\deg x_1 < m \\ (x_1, F) = 1} }\gamma_{x_1}e_F(a\ov{x}_1\lambda )\Bigg{|}^{k_1}. 
\end{align*}
Applying the Cauchy-Schwarz inequality now yields
\begin{align*}
   S^{2k_1k_2} 
   &\ll_{\epsilon} q^{2\epsilon r k_1k_2}~q^{2mk_1(k_2-1)+ 2nk_2(k_1-1)} \\
   &\hspace{5em}\times \sum_{\deg \lambda < r}I_{F, \lambda, k_2}(n)^2 \times \sum_{\deg \lambda < r}\Bigg{|} \sum_{\substack{\deg x_1 < m \\ (x_1, F) = 1}} \gamma_{x_1}e_F(a\ov{x}_1\lambda)\Bigg{|}^{2k_1}\\
   &\ll q^{2\epsilon r k_1k_2}~q^{2mk_1(k_2-1)+ 2nk_2(k_1-1)} ~ E^{\inv}_{F, k_2}(n)~q^rE_{F, k_1}^{\inv}(m)
\end{align*}
and rearranging gives the desired result. 
\end{proof}
Another useful way to state Lemma \ref{lem:bilinear} is 
\begin{equation}\label{eq:bilinear_rearranged}
    \begin{aligned}
         \cW_{F,a}(m,n;\balpha, &\bbeta)~\ll_{\epsilon} ~q^{m+n + \epsilon r}\\
    &\times\bigg{(}E_{F,k_1}(m)q^{r/2-2mk_1}\bigg{)}^{\frac{1}{2k_1k_2}}\bigg{(}E_{F,k_2}(n)q^{r/2-2nk_2}\bigg{)}^{\frac{1}{2k_1k_2}}. 
    \end{aligned}
\end{equation}
  
A simpler result is the following, which is obtained using the argument of \cite[Lemma 2.4]{garaev2010}
\begin{lem}\label{lem:garaev_general}
   Let $k$ denote a positive integer and take other notation as in Lemma \ref{lem:bilinear}. Then 
\begin{align*}
      \cW_{F,a}(m,n;\balpha, \bbeta)
      &\ll_\epsilon q^{(m(2k-1)+\max\{r, m\})/2k + \epsilon r}E_{F,k}^\inv(n)^{1/2k}.
\end{align*} 
\end{lem}
\begin{proof}
Again let $S = |\cW_{F,a}(m,n; \balpha, \bbeta)|$. Applying H{\"o}lders inequality and rearranging yields 
\begin{align*}
    S^{2k} 
    &\ll_\epsilon q^{k\epsilon r}q^{m(2k-1)}\sum_{\substack{\deg x_1 < m \\ (x_1, F) = 1}}\Bigg{|} \sum_{\substack{\deg x_2 < n \\ (x_2, F) = 1}}\beta_{x_2}e_F(a\ov{x_1}\ov{x_2})\Bigg{|}\\
    &\ll_\epsilon q^{k\epsilon r}q^{m(2k-1) + \max\{0, m-r\}}\sum_{\substack{\deg x_1 < r}}\Bigg{|} \sum_{\substack{\deg x_2 < n \\ (x_2, F) = 1}}\beta_{x_2}e_F(ax_1\ov{x_2})\Bigg{|}^{2k}. 
\end{align*}
Expanding and using orthogonality then implies 
\begin{align}\label{eq:garaev_additive_halfway}
    S^{2k} 
    &\ll_\epsilon q^{2k\epsilon r}q^{m(2k-1)+ \max\{0, m-r\} + r}E_{F, k}^\inv(n)
\end{align}
as desired. 

\end{proof}

In the case of $k=2$, this becomes 
\begin{align}\label{eq:garaev}
      \cW_{F,a}(m,n;\balpha, \bbeta)
      &\ll_\epsilon q^{(3m+\max\{r, m\})/4 + \epsilon r}E_{F,2}^\inv(n)^{1/4}. 
\end{align}
which will be used most often. 

We can again improve upon this by averaging over the modulus, using the exact same ideas as in the proof of Lemma \ref{lem:garaev_general} above. 
\begin{lem}\label{lem:garaev_average}
    With notation as in Lemma \ref{lem:garaev_general},
    \begin{align*}
        &\sum_{\deg F = r}\max_{(a,F) = 1} |\cW_{F,a}(m,n;\balpha, \bbeta)| \\
        &\qquad\qquad\qquad\ll_\epsilon q^{(m+r)\frac{2k-1}{2k} + \max\{m, r\}\frac{1}{2k} + \epsilon r}\left(\sum_{\deg F = r} E_{F,k}^\inv(n)\right)^{1/2k} 
    \end{align*}
\end{lem}
\begin{proof}
We can use equation \eqref{eq:garaev_additive_halfway} and then H{\"o}lders inequality to see
\begin{align*}
&\sum_{\deg F = r}\max_{(a,F) = 1} |\cW_{F,a}(m,n;\balpha, \bbeta)| \\
&\qquad\qquad\qquad\ll_{\epsilon} q^{\epsilon r}q^{m\frac{2k-1}{2k} + \max\{m, r\}\frac{1}{2k}}\sum_{\deg F = r}E_{F,k}^\inv(n)^{\frac{1}{2k}}\\
&\qquad\qquad\qquad \ll q^{\epsilon r}q^{m\frac{2k-1}{2k} + \max\{m, r\}\frac{1}{2k}}\left(\sum_{\deg F = r}E_{F,k}^\inv(n)\right)^{\frac{1}{2k}}q^{r\frac{2k-1}{2k}}
\end{align*}
and rearranging gives the desired result. 
\end{proof}

\subsection{Proof of Theorem \ref{thm:bilinear_savings}}
As before we let $S = |\cW_{F,a}(m,n; \balpha, \bbeta)|$, and split the discussion into a few cases. Without loss of generality, we may suppose that $n \leq m$.

First, we assume that $n \leq r/3$, and let $k \geq 2$ denote the largest integer such that $n(k-1) \leq r/2$. Note that $k$ is bounded above in terms of $\epsilon$ since $n$ is from below, and $nk > r/2$. Thus applying \eqref{eq:bilinear_rearranged} with $k_1 = 2$ and $k_2 = k$, together with Lemma \ref{lem:inverse_energy} gives
\begin{align*}
    S 
    &\ll_{\epsilon'} q^{m+n+\epsilon' r}\bigg{(}E_{F,2}(m)q^{r/2-4m}\bigg{)}^{\frac{1}{4k}}\bigg{(}q^{r/2- nk} + q^{n(k-1)-r/2} \bigg{)}^{\frac{1}{4k}}\\
    &\ll q^{m+n+\epsilon' r}\bigg{(}E_{F,2}(m)q^{r/2-4m}\bigg{)}^{\frac{1}{4k}}.
\end{align*}
for some sufficiently small $\epsilon '$. Since $k$ is bounded from above, it now suffices to show that for any $m > r(1/4 + \epsilon)$, 
$$E_{F,2}(m)q^{r/2-4m} < q^{-\delta_1r}$$
for some $\delta_1 > 0$. If $ r(1/4 + \epsilon) < m < r/3$, then Lemma \ref{lem:inverse_energy_k=2} yields 
$$E_{F,2}(m)q^{r/2-4m} \ll_{\epsilon} q^{r/2-2m + \epsilon m} < q^{-r\epsilon}$$
as desired. Similarly, applying Lemma \ref{lem:inverse_energy_k=2} in the case $r/3 \leq m \leq r$ and the case $r \leq m$ gives the desired result when $n \leq r/3$.

Next, we may assume $m,n \geq r/3$. By Lemma \ref{lem:bilinear} with $k_1=k_2=2$, it suffices to show 
\begin{align*}
    E_{F,k_1}(m)E_{F,k_2}(n)q^{r-4n-4m} < q^{-\delta_2r}
\end{align*}
for some $\delta_2 > 0$. If $r/3 \leq n \leq m \leq r$, then Lemma \ref{lem:inverse_energy_k=2} gives 
\begin{align*}
    E_{F,k_1}(m)E_{F,k_2}(n)q^{r-4n-4m} \ll_{\epsilon} q^{-m/2-n/2 + \epsilon m}
\end{align*}
which is sufficient. Similarly, applying Lemma \ref{lem:inverse_energy_k=2} in the cases 
$r/2 \leq n \leq r$ and $r \leq m$, as well as $r \leq n \leq m$, both yield the desired result.

\section{Applications}
Before proceeding we will make a few reductions common to each of Theorems \ref{thm:mobius_r/2}, \ref{thm:mobius_3/4} and \ref{thm:irving}. For $a,F \in \F_q[T]$ with $\deg F = r$ we set
\begin{equation}
    \begin{aligned}\label{eq:F_0_a_0_def}
        d = \deg(a,F), ~F_0 = F/(a,F), ~r_0 = \deg F_0, ~a_0 = a/(a,F).
    \end{aligned}
\end{equation}

\begin{lem}\label{lem:mobius_bound_smallr0}
\begin{align*}
    \sum_{\substack{\deg x < n \\ (x,F) = 1}}\mu(x)e_F(a\ov{x})
    &\ll_{\epsilon} q^{r_0+ n(1/2 +\epsilon)}. 
\end{align*}
\end{lem}
\begin{proof}
    This is a direct application of Corollary \ref{cor:mobius_arith} as 
    \begin{align*}
         \sum_{\substack{\deg x < n \\ (x,F) = 1}}\mu(x)e_F(a\ov{x})
    &= \sum_{\substack{\deg x < n \\ (x,F) = 1}}\mu(x)e_{F_0}(a_0\ov{x}) \\
    &= \sum_{\substack{\deg y < \deg F_0 \\ (y,F) = 1}}e_{F_0}(a\ov{y})\sum_{\substack{\deg x < n \\ x \equiv y \Mod{F_0}}}\mu(x) \ll_\epsilon q^{r_0+ n(1/2 +\epsilon)}. 
    \end{align*}
\end{proof}

\begin{lem}\label{lem:mobius_bound_bigr0}
For any positive integer $U$ satisfying $2U < n$, 
\begin{align*}
     \Bigg{|}\sum_{\substack{\deg x < n \\ (x,F) = 1}}\mu(x)e_F(a\ov{x})\Bigg{|}
    \ll_\epsilon S_1 + S_2
    \end{align*}
    where 
   \begin{align*} 
    S_1 = q^{n\epsilon}\sum_{\substack{\deg x \leq u \\ (x,F) = 1}}\bigg{|}\sum_{\substack{\deg y < n-u \\ (y,F) = 1}}e_{F_0}(a_0\ov{xy})\bigg{|}, ~S_2 = q^{n\epsilon}\sum_{\substack{\deg x \leq v \\ (x,F) = 1}}\bigg{|} \sum_{\substack{\deg y < n-v \\ (y,F) = 1}}\beta_ye_{F_0}(a_0\ov{xy}) \bigg{|}.
\end{align*}
for some integers $u \leq 2U$ and $U < v \leq n-U$, and $|\beta_y| \ll_\epsilon q^{n\epsilon}$.
\end{lem}
\begin{proof}
This follows from a standard manipulation of Vaughan's identity as in \cite{garaev2010, fouvryshparlinski2011, BourgainGaraev2013, irving2014}, so we will not include the details here. For this setting see the appendix in \cite{sawinshusterman}.

\end{proof}

\subsection{Proof of Theorem \ref{thm:mobius_r/2}}
Recall that we fix $\epsilon > 0$ and suppose that $r(1/2 + \epsilon)<  n  < r$. Additionally, recall $r_0, F_0$ and $a_0$ from \eqref{eq:F_0_a_0_def}. 

If $r_0 \leq n(1/2 - 2\epsilon)$ then Lemma \ref{lem:mobius_bound_smallr0} implies 
$$\sum_{\substack{\deg x < n \\ (x,F) = 1}}\mu(x)e_F(a\ov{x}) \ll_{\epsilon} q^{n(1-\epsilon)} $$
as desired. So we may now assume $r_0 > n(1/2 - 2\epsilon)$. 

By letting 
$$U = (n-r/2-r\epsilon/2)/2,$$
it suffices to bound the sums $S_1$ and $S_2$ in Lemma \ref{lem:mobius_bound_bigr0}. 

First we bound $S_1$ by applying Lemma \ref{lem:weil_incomplete} and Lemma \ref{lem:weil_different_modulus}. If $n-u \geq r_0$ then Lemma \ref{lem:weil_different_modulus} implies
$$S_1 \ll_{\epsilon} q^{u + n-u - r_0 + r_0/2 + \epsilon r/2} =q^{n-r_0/2 + \epsilon r/2}.$$
If $n-u \leq r_0$ then by $u \leq 2U = n-r/2-r\epsilon/2$ and Lemma \ref{lem:weil_incomplete},
$$S_1 \ll_{\epsilon} q^{u + r_0/2 + r\epsilon/3} \leq q^{n-r\epsilon/2 + n\epsilon/3}. $$
Either way, these provide sufficient power savings. 

Finally to bound $S_2$, we can directly apply Theorem \ref{thm:bilinear_savings} which completes the proof.

\subsection{Proof of Theorem \ref{thm:mobius_3/4}}
This proof is quite similar to the proof of Theorem \ref{thm:mobius_r/2}, and just requires slightly more attention to detail to obtain more explicit bounds. This expands upon some ideas from \cite{garaev2010, BanksHarcharrasShparlinski2003}. Again, recall the notation $a_0, r_0$ and $F_0$ from \eqref{eq:F_0_a_0_def}. We may assume $n > 3r/4$ since otherwise, the result is trivial. 

First, suppose that $r_0 < 7n/16$. Then Lemma \ref{lem:mobius_bound_smallr0} implies 
\begin{align*}
     \Bigg{|}\sum_{\substack{\deg x < n \\ (x,F) = 1}}\mu(x)e_F(a\ov{x})\Bigg{|}
    &\ll_\epsilon q^{ n(15n/16 +\epsilon)}
\end{align*}
as desired, and thus we may now assume $r_0 \geq 7n/16$. 

By letting $U =r_0/3$, we need to bound $S_1$ and $S_2$ as in Lemma \ref{lem:mobius_bound_bigr0}. First, we deal with $S_1$ and split the argument up into cases depending on the sizes of $u$ and $n-u$. 

\textit{Case 1:} $u \leq 2r_0/3$ and $r_0 \leq n-u$. Here we apply Lemma \ref{lem:weil_different_modulus} to the inner sum over $y$ to obtain 
\begin{align*}
    S_1
    &\ll_\epsilon q^{n-u-r_0 + \epsilon n/2}\sum_{\substack{\deg x \leq u \\ (x,F) = 1}}\bigg{|}\sum_{\substack{\deg y < r_0 \\ (y,F) = 1}}e_{F_0}(a_0\ov{xy}) \bigg{|}\\
    &\ll_{\epsilon} q^{n-r_0/2 + \epsilon n} \leq q^{25n/32 + \epsilon n}
\end{align*}
since $r_0 \geq 7n/16$. 

\textit{Case 2:} $u \leq r_0/3$ and $r_0/3 \leq n-u \leq r_0$. Using equation \eqref{eq:garaev} together with Lemma \ref{lem:inverse_energy_k=2} yields 
$$S_1 \ll_{\epsilon} q^{3(n-u)/4 + r_0/4 + u/2 + \epsilon n} = q^{3n/4 + r_0/4 - u/4 + \epsilon n }.$$
Also, separately applying Lemma \ref{lem:weil_different_modulus} to $S_1$ (to the inner sum over $y$) implies 
$$S_1 \ll_{\epsilon} q^{u + r_0/2 + \epsilon n}.$$
By combining these two estimates we have 
$$S_1 \ll_{\epsilon} q^{3n/5 + 3r_0/10 + \epsilon n} \leq q^{2n/3 + r/4 + \epsilon n}$$
since $n > 3r/4$. 

\textit{Case 3:} $r_0/3 \leq u \leq /3$ and $r_0/3 \leq n-u \leq r_0$. Here we use  Lemma \ref{lem:bilinear} with $k_1=k_2=2$ together with Lemma \ref{lem:inverse_energy_k=2}, giving 
$$S_1 \ll_{\epsilon} q^{n + \epsilon n + \frac{1}{8}(7u/2 - r_0/2 + 7(n-u)/2 - r_0/2 + r_0 - 4t_1 - 4(n-u))} = q^{15n/16 + \epsilon n}.$$

For the remaining cases bounding $S_1$, we may assume that $n-u \leq r_0/3$, which implies $u \geq n-r_0/3$. Note that this also implies $n \leq r_0$ since $u \leq 2r_0/3$. 

\textit{Case 4:} $n-r_0/3 \leq u \leq 2n/3$ and $n-u \leq r_0/3$. Here, by again applying equation \eqref{eq:garaev} with Lemma \ref{lem:inverse_energy_k=2} we have 
$$S_1 \ll_{\epsilon} q^{3u/4 + r_0/4 + (n-u)/2 + \epsilon n} \leq q^{2n/3 + r/4 + \epsilon n}.$$

\textit{Case 5:} $2n/3 \leq u \leq 2r_0/3$ and $n-u \leq r_0/3$. Applying the Cauchy-Schwarz inequality directly to $S_1$ shows
\begin{align*}
    S_1^2
    &\ll_{\epsilon} q^{u + \epsilon n}\sum_{\substack{y_1,y_2 \\ \deg y_i < n-u \\ (y_i,F) = 1}}\sum_{\substack{\deg x < u \\ (x,F) = 1}}e_F(a\ov{x}(\ov{y_1}+\ov{y_2})).
\end{align*}
Isolating the case $y_1 = -y_2$ and then applying Lemma \ref{lem:weil_incomplete} to the sum over $x$ implies 
\begin{align}\label{eq:3r/4_garaevcase}
    S_1^2
    &\ll_{\epsilon} q^{n + u+\epsilon n } + q^{u + r/2 + \epsilon n}T
\end{align}
where 
$$T = \sum_{\substack{y_1,y_2 \\ \deg y_i < n-u \\ (y_i,F) = 1}}q^{\deg(F,\ov{y_1} + \ov{y_2})/2}.$$
We can rearrange and write
\begin{align*}
  T =  \sum_{\substack{d|F \\ d\text{ monic}}}q^{\deg d/2}\sum_{\substack{0 \leq \deg a < r \\ (a,F) = d}}I_{F,a}(n-u)
\end{align*}
with $I_{F,a}(n-u)$ as in \eqref{eq:sum_of_inv}. Now mimicking the argument after equation \eqref{eq:E_inv_split_gcd_2} identically shows
$$T \leq \sum_{\substack{d|F \\ d\text{ monic}}}q^{\deg d/2}q^{2n-2u - \deg d} \ll_{\epsilon} q^{2n-2u + \epsilon n}. $$
Substituting back into \eqref{eq:3r/4_garaevcase} yields
$$S_1 \ll_{\epsilon} q^{n/2 + u/2 + \epsilon n } + q^{r/4 +n-u/2 +\epsilon n} \ll q^{2n/3 + r/4 + \epsilon n}$$
where we have again used $n > 3r/4$. 

Combining all $5$ cases above yields a suitable bound for $S_1$. We now focus on bounding $S_2$, and similarly consider a number of cases depending on the size of $v$ and $n-v$. Without loss of generality, we may assume that $v \leq n-v$. 

\textit{Case 1:} $r_0/3 \leq v \leq r_0$ and $r_0/3 \leq n-v \leq r_0$. Here we may apply bounds identically to \text{Case 3} above when bounding $S_1$. 

The last two cases both use equation \eqref{eq:garaev} with Lemma \ref{lem:inverse_energy_k=2}.

\textit{Case 2:} $r_0/3 \leq v \leq r_0$ and $r_0 \leq n-v $. Here
$$S_2 \ll_{\epsilon} q^{n-v + 7v/8 - r_0/8 + \epsilon n} \leq  q^{15n/16 + \epsilon n}$$
since $r_0 \geq 7n/16$. 

\textit{Case 3:} $r_0 \leq v$ and $r_0\leq n-v$. In this case
$$S_2 \ll_{\epsilon} q^{v + n-v - r_0 + \epsilon n} \leq q^{n - r_0/4 + \epsilon n} \leq q^{15n/16 + \epsilon n}$$
since $r_0 \geq 7n/16$. 

Combining these cases yields a suitable bound for $S_2$, which now completes the proof.

\subsection{Proof of Theorem \ref{thm:irving}}
Again, this proof is similar to the other proofs previously in this section. Recall that we are wanting to bound 
$$S = \sum_{\deg F = r}\max_{a \in \F_q[T]}\bigg{|}\sum_{\substack{\deg x < n \\ (x,F) = 1}}\mu(x)e_F(a\ov{x})\bigg{|}.$$
For each $F$, let $a_F$ denote the value of $a$ for which the maximum on the inner sum is achieved. Then we can say 
\begin{align}\label{eq:irving_reduction}
    S
    &= \sum_{\substack{\deg d \leq r \\ d \text{ monic}}}\sum_{\substack{\deg F = r \\ (a_F, F) = d}}\bigg{|}\sum_{\substack{\deg x < n \\ (x,F) = 1}}\mu(x)e_{F/d}((a_F/d)\ov{x})\bigg{|}\nonumber\\
    &\leq \sum_{\substack{\deg d \leq r \\ d \text{ monic}}}\sum_{\substack{\deg F = r-\deg d }}\max_{(a,F) = 1}\bigg{|}\sum_{\substack{\deg x < n \\ (x,F) = 1}}\mu(x)e_{F}(a\ov{x})\bigg{|} \nonumber\\
    &= \sum_{j=1}^{r}q^j\sum_{\deg F = r-j}\max_{(a,F) = 1}\bigg{|}\sum_{\substack{\deg x < n \\ (x,F) = 1}}\mu(x)e_{F}(a\ov{x})\bigg{|}\nonumber\\
    &\ll_{\epsilon} q^{\epsilon n + j}\sum_{\deg F = r-j}\max_{(a,F) = 1}\bigg{|}\sum_{\substack{\deg x < n \\ (x,F) = 1}}\mu(x)e_{F}(a\ov{x})\bigg{|}
\end{align}
for some integer $1 \leq j \leq r$. 

First, suppose $j > r- 2n/5$. Then applying Lemma \ref{lem:mobius_bound_smallr0} to the inner-sum implies 
\begin{align*}
    q^{\epsilon n + j}\sum_{\deg F = r-j}\max_{a \in \F_q[T]}\bigg{|}\sum_{\substack{\deg x < n \\ (x,F) = 1}}\mu(x)e_F(a\ov{x})\bigg{|} &\ll_{\epsilon} q^{j + r-j + r-j + n/2 + n\epsilon}\\
&\ll q^{r + 9n/10 + n\epsilon}
\end{align*}
so we may assume that $j\leq r-2n/5$. 

We let $U = \min\{n/3, 5n/8 - r/4 \}$. By Lemma \ref{lem:mobius_bound_bigr0} and equation \eqref{eq:irving_reduction} the problem reduce to bounding
$$T_1 = q^{\epsilon n + j}\sum_{\deg F = r-j}\max_{(a,F) = 1}S_1, ~~T_2 = q^{\epsilon n + j}\sum_{\deg F = r-j}\max_{(a,F) = 1}S_2$$
with $S_1$ and $S_2$ as in Lemma \ref{lem:mobius_bound_bigr0}. In Lemma 
\ref{lem:mobius_bound_bigr0}, the condition on $u$ is given as $u \leq 2U$. But since $2U \leq n-U$ here, the case of $U \leq u \leq 2U$ is covered when dealing with $S_2$ (since all of our methods for bounding $S_2$ also apply to $S_1$). So when bounding $S_1$ we may assume $u \leq U$. 

First, we deal with $T_1$. We may apply Lemma \ref{lem:weil_incomplete} to the inner sum over $y$. If $n-u \leq r-j$ then
$$T_1 \ll_{\epsilon} q^{n\epsilon + j + r-j + u + (r-j)/2} \ll q^{3r/2 + u + n\epsilon } \ll q^{5n/8 + 5r/4 + n\epsilon},$$
or if $n-u \geq r-j$ then 
$$T_1 \ll_{\epsilon} q^{n\epsilon + j + r-j + u +n-u-(r-j) + (r-j)/2} \ll q^{r + 4n/5 + n\epsilon } $$
where we have used $j \leq r-2n/5$. 

Next, we deal with $T_2$. By Lemma \ref{lem:garaev_average} and Lemma \ref{lem:add_energy_average} we have  that for any positive integer $k$, 
\begin{align}\label{eq:T_2}
    T_2
    &\ll_{\epsilon} q^{\epsilon n + j}q^{(v + r-j)\frac{2k-1}{2k} + \max\{v, r-j\}\frac{1}{2k}}\left(q^{(r-j)/2k + n/2 - v/2} + q^{n-v} \right).
\end{align}
We consider two cases depending on the size of $v$ and $n-v$. Since we have treated the inner sum $S_2$ in $T_2$ as a bilinear Kloosterman sum with arbitrary weights, and the ranges on $v$ and $n-v$ are equal, we may also interchange $v$ and $n-v$. Thus, considering the range $2n/5 \leq v \leq n-U$ is enough, since if $v \leq 2n/5$ then $n-v \geq 3n/5$, so we may swap $v$ and $n-v$ to get back into the range $2n/5 \leq v \leq n-U$. 

\textit{Case 1:} $2n/5 \leq v \leq n/2$ and $\max\{v, r-j\} = r-j$. Here we use \eqref{eq:T_2} with $k=2$,  
\begin{align*}
    T_2
    &\ll_{\epsilon} q^{r + n\epsilon}(q^{n-v/4} + q^{n/2+r/4-j/4+v/4})\ll_{\epsilon} q^{r + n\epsilon}(q^{9n/10} + q^{r/4 +5n/8}).
\end{align*}

\textit{Case 2:} $3n/5 \leq v \leq n-U$ and $\max\{v, r-j\} = r-j$. Here we use \eqref{eq:T_2} with $k=3$,  
\begin{align*}
    T_2
    &\ll_{\epsilon} q^{r + n\epsilon}(q^{r/6 -j/6 + n/2 + v/3} + q^{n-v/6 })\\
    &\ll q^{r + n\epsilon}(q^{r/6 + 5n/6 - U/3} + q^{9n/10})\\
    &\ll q^{r + n\epsilon}(q^{13n/18 + r/6} + q^{5n/8 + r/4} + q^{9n/10})
\end{align*}
where we have used $U = \min\{n/3, 5n/8-r/4\}$. 

\textit{Case 3:} $2n/5 \leq v \leq n - U$ and $\max\{v, r-j\} = v$. Here we use \eqref{eq:T_2} with $k=2$,  
\begin{align*}
    T_2
    &\ll_{\epsilon} q^{r + n\epsilon}(q^{n/2 + v/2} + q^{n-r/4 + j/4})\\ 
    &\ll q^{r + n\epsilon}(q^{5n/6}+ q^{11n/16 + r/8}  + q^{9n/10})
\end{align*}
where we have used $j \leq r-2n/5$ and $U = \min\{n/3, 5n/8-r/4\}$.

Combining all of our estimates for $T_1$ and $T_2$ yields the desired result. 

\subsection{Proof of Theorem \ref{cor:vonmangoldt}}
This result follows from substituting Theorem \ref{thm:mobius_3/4} instead of \cite[Theorem 1.8]{sawinshusterman} into the proof of \cite[Theorem 1.9]{sawinshusterman}, but we sketch the details here. We let $d = n -r$, and thus the condition that 
$n\omega \geq r$ for some $\omega < 1/2 + 1/62$ can be rewritten as
$$d \geq r\frac{1-\omega}{\omega} = r(1-\omega')$$
for some $\omega' < 1/16$. We let $\theta > 0$ (which will be taken to be sufficiently small as needed). Also, we let 
$$\epsilon = \frac{16}{15}\left(\frac{1}{16}-\omega'-2\theta\right).$$
By \cite[equation (5.9)]{sawinshusterman}, it suffices to bound 
\begin{align*}
    S 
    &= \sum_{k=1}^{d+r}k\sum_{\substack{x \in \cM_k \\ (x,F) = 1}}\mu(x)\sum_{\substack{y \in \cM_{r+d-k} \\ xy \equiv a \Mod{F}}}1. 
\end{align*}
As in \cite{sawinshusterman}, if $k \leq d $, we can apply Lemma \ref{lem:sawinshusterman_mainterm} to contribute the main term. 

We denote the remaining sum over $k > d$ by $S_0$, and note that 
$$S_0 \leq rk\Bigg{|}\sum_{\substack{x \in \cM_{k_0}\\ (x,F) = 1}}\mu(x)\sum_{\substack{y \in \cM_{r+d-k_0} \\ xy \equiv a \Mod{F}}}1\Bigg{|}$$
for some $k$ satisfying $d \leq k \leq d+r$. If $k \leq r(1+\epsilon)$, then using \cite[equation (5.10)]{sawinshusterman}, applying Theorem \ref{thm:mobius_3/4}  and using $k \leq r(1+\epsilon)$ yields 
\begin{align}\label{eq:BV_to_mobius}
    S_0
    &\leq rkq^{d-k}\sum_{\substack{\deg h < k-d}}\Bigg{|}\sum_{\substack{x \in \cM_{k} \\ (x,F) = 1}}\mu(x)e_F(ah\ov{x}) \Bigg{|}\\
    &\ll_\theta rkq^{\theta r}(q^{15k/16} + q^{2k/3 + r/4})\nonumber\\
    &\ll_{\theta} q^{15r/16 + 15r\epsilon/16 + r\theta}. \nonumber
\end{align}
We now use $\epsilon = (1/16 - \omega' - 2\theta)16/15$ and then $r \leq d +r\omega'$ to conclude
\begin{align*}
    S_0
    &\ll_{\theta} q^{r-r\omega - r\theta} \leq q^{d - r\theta}
\end{align*}
which is sufficient. So we may now assume that $k > r(1+\epsilon)$. We also let $\beta >0$. Rearranging $S_0$ we arrive at \cite[equation (5.13)]{sawinshusterman}, 
$$S_0 \leq rk\sum_{\substack{y \in \cM_{r+d-k_0}}}\bigg{|}\sum_{\substack{x \in \cM_{k} \\ xy \equiv a \Mod{F}}}\mu(x)  \bigg{|}.$$
Thus we may apply Lemma \ref{lem:sawinshusterman_mobius_dist} to yield 
$$S_0 \ll_{\theta, \beta} rkq^{r+d-k}q^{(k-r)(1-\beta/p)} \ll_{\beta} q^{d-r\beta/p\epsilon}$$
which again, is sufficient. This holds as long as 
$$q > \left(pe\frac{\epsilon + 2}{\epsilon} \right)^{\frac{2}{1-2\beta}} = \left(pe\left(1 + \frac{30}{1-16\delta'-32\theta}\right) \right)^{\frac{2}{1-2\beta}}.$$
But since we fix $p$ and $q$, we may choose $\theta$ and $\beta$ sufficiently small so that we only require 
\begin{align}\label{eq:q_bound_for_twin_primes}
    q > p^2e^2\left(1 + \frac{30}{1-16\omega'}\right)^{2}.
\end{align}
By substituting $\omega$ and rearranging we obtain the desired result.

\subsection{Proof of Theorem \ref{thm:bombierivinogradov}}
This proof uses essentially the same ideas as the proof of Theorem \ref{cor:vonmangoldt}, although it is slightly more technical. We may assume that $n = O(r)$, since for small $r$ this is implied by other results (for example, by Theorem \ref{cor:vonmangoldt}). We let $d = n-R$, and thus the condition that $R \leq n\omega$ for some $\omega < 1/2 + 1/38$ can be rewritten as $d \geq R(1-\omega')$ for some $\omega' < 1/10$. We let $\theta > 0$ (which will be taken to be sufficiently small as needed) 
Also, we let
$$\epsilon = \frac{10}{9}\left(\frac{1}{10}-w'-2\theta\right).$$

We rewrite the sum in question as
$$S = \sum_{r = 1}^{R-1}\sum_{\deg F = r}\max_{(a,F) = 1}\bigg|\sum_{\substack{x \in \cM_n \\ x \equiv a \Mod{F}}}\Lambda(x) ~-\frac{q^n}{\phi(F)} \bigg|.$$

For each $r$ in this sum, let $d_r = n-r$. Expanding this identically as in \cite[equation (5.9)]{sawinshusterman}, we can say $S \ll S_1 + S_2 + S_3$
where
\begin{align*}
    S_1
    &= \sum_{r=1}^{R-1}\sum_{\deg F = r}\bigg{|}-q^{d_r}\sum_{k=1}^{d_r}kq^{-k}\sum_{\substack{x \in \cM_k \\ (x,F) = 1}}\mu(x) -\frac{q^n}{\phi(F)}\bigg{|},\\
    S_2 &= \sum_{r=1}^{R-1}\sum_{\deg F = r}\max_{(a,F) = 1}\sum_{\substack{d_r < k < d_r + r \\ k \leq r(1+\epsilon)}}k\bigg{|}\sum_{\substack{x \in \cM_k \\ (x,F) = 1}}\mu(x)\sum_{\substack{y \in \cM_{r+d_r-k} \\ xy \equiv a \Mod{F}}}1\bigg{|},\\
    S_3 &= \sum_{r=1}^{R-1}\sum_{\deg F = r}\max_{(a,F) = 1}\sum_{\substack{d_r < k < d_r + r \\ k > r(1+\epsilon)}}k\bigg{|}\sum_{\substack{x \in \cM_k \\ (x,F) = 1}}\mu(x)\sum_{\substack{y \in \cM_{r+d_r-k} \\ xy \equiv a \Mod{F}}}1\bigg{|},
\end{align*}
and it suffices to show each $S_i \ll_{\omega} q^{n-R\delta}$ for some $\delta > 0$. 

To bound the contribution from $S_1$, we apply Lemma \ref{lem:sawinshusterman_mainterm} directly to see
\begin{align*}
    S_1
    &= \sum_{r=1}^{R-1}\sum_{\deg F = r}\left|q^{d_r}\left(-\frac{q^r}{\phi(F)} + q^{o(n)-d_r} \right) + \frac{q^n}{\phi(F)}\right|\\
    &\ll_{\theta} q^{R + \theta n} < q^{n-R\delta}
\end{align*}
for some $\delta >0$, since we can choose $\theta$ and $\delta$ sufficiently small and $R < n$.

Next, we consider $S_2$. Identically as in equation \eqref{eq:BV_to_mobius}, we can use Lemma \ref{lem:orthogonality} to say
\begin{align*}
    S_2 
    &\ll \sum_{r=1}^{R-1}\sum_{\substack{d_r < k < d_r + r \\ k \leq r(1+\epsilon)}}kq^{d_r-k}\sum_{\substack{\deg h < k-d_r}}\sum_{\deg F = r}\max_{(a,F) = 1}\Bigg{|}\sum_{\substack{x \in \cM_{k} \\ (x,F) = 1}}\mu(x)e_F(ah\ov{x}) \Bigg{|}.
\end{align*}
Then applying Theorem \ref{thm:irving} and using $k \leq r(1+\epsilon)$  and $r \leq R$ yields 
\begin{align*}
    S_2
    &\ll_{\theta} \sum_{r=1}^{R-1}\sum_{\substack{d_r < k < d_r + r \\ k \leq r(1+\epsilon)}}k\left(q^{5r/4 + 5k/8} + q^{r + 9k/10} + q^{7r/6 + 13k/18} \right)q^{\theta r/2}\\
    &\ll_{\theta} q^{R + 9R/10 + 9R\epsilon/10 + \theta R}
\end{align*}
where here we have used $kR \ll_{\theta} q^{R\theta/2}$. Using $\epsilon = (1/10-\omega'-2\theta)10/9$ and $R \leq d + R\omega' =  n-R + R\omega'$ means 
\begin{align*}
    S_2
    &\ll_{\theta} q^{n-R\theta}
\end{align*}
as desired. 

Finally, to bound $S_3$ we let $ \beta, \beta' > 0$ (which we will take to be sufficiently small as needed) and we can apply Lemma \ref{lem:sawinshusterman_mobius_dist}. Note that working identically to equation \eqref{eq:q_bound_for_twin_primes}, this will only hold for 
$$q > p^2e^2\left(1+\frac{18}{1-10\omega'} \right)^2.$$
Regardless, regarranging $S_3$ and then applying Lemma \ref{lem:sawinshusterman_mobius_dist} means 
\begin{align*}
        S_3 &= \sum_{r=1}^{R-1}\sum_{\deg F = r}\max_{(a,F) = 1}\sum_{\substack{d_r < k < d_r + r \\ k > r(1+\epsilon)}}k\sum_{\substack{y \in \cM_{r+d_r-k} \\ (y,F) = 1}}\bigg{|}\sum_{\substack{x \in \cM_k \\ x \equiv \ov{y}a \Mod{F}}}\mu(x)\bigg{|}\\
        &\ll_{\beta, \beta', \theta} \sum_{r=1}^{R-1}\sum_{\substack{d_r < k < d_r + r \\ k > r(1+\epsilon)}}q^{n-\beta/p(k-r) +  n\beta'}
\end{align*}
where we have used $k \ll_{\beta'} q^{n\beta'}$. 
We now deal with two parts of this sum separately. For $r < R/3$ (which means $r < n/3$), we make the substitution $k > d_r = n-r$ to give 
\begin{align*}
    \sum_{r=1}^{R/3}\sum_{\substack{d_r < k < d_r + r \\ k > r(1+\epsilon)}}q^{n-\beta/p(k-r) +  n\beta'}
    &\ll_{\beta'} \sum_{r=1}^{R/3}q^{n-\beta/p(n-2r) + 2n\beta'}\\
    &\ll_{\beta'} q^{n- n(\beta/(3p) - 3\beta')}\\
    &\ll q^{n- R(\beta/(3p) - 3\beta')}
\end{align*}
which is admissable for $\beta$ and $\beta'$ chosen suitably. Finally for $r \geq R/3$, we make the substitutions $k > r(1+\epsilon)$, $\epsilon = (1/10 - \omega' - 2\theta)10/9$ and $d_r = n-r$ to give 
\begin{align*}
    \sum_{r=R/3}^{R-1}\sum_{\substack{d_r < k < d_r + r \\ k > r(1+\epsilon)}}q^{n-\beta/p(k-r) +  n\beta'}
    &\ll_{\beta'} \sum_{r=R/3}^{R-1}q^{n-r\epsilon\beta/p +  2n\beta'}\\
    &\ll_{\beta'} q^{n-R\epsilon\beta/(3p) +  3n\beta'} 
\end{align*}
which is admissible, since we have assumed that $n = O(r)$ and we may choose $\beta, \beta'$ suitably. 

Combining our estimates for $S_1, S_2$ and each part of $S_3$ gives the result.

\section{Acknowledgements}
The author is very grateful to Bryce Kerr and Igor Shparlinski for many long and helpful discussions about this work, and for reading over multiple drafts of this paper. 
During the preparation of this work, the author was supported by an Australian Government Research Training Program (RTP) Scholarship.

\bibliographystyle{plain} 

\bibliography{refs}

\begin{thebibliography}{10}

\bibitem{Bagshaw2023}
C.~Bagshaw.
\newblock Bilinear forms with {K}loosterman and {G}auss sums in function fields.
\newblock \url{https://arxiv.org/abs/2304.05014}, 2023.

\bibitem{Bagshaw2023_PAMS}
C.~Bagshaw.
\newblock Square-free smooth polynomials in residue classes and generators of irreducible polynomials.
\newblock {\em Proc. Amer. Math. Soc.}, 151(03):1017--1029, 2023.

\bibitem{BagshawKerr2023}
C.~Bagshaw and B.~Kerr.
\newblock Lattices in function fields and applications.
\newblock \url{https://arxiv.org/abs/2304.05009}, 2023.

\bibitem{BS2022}
C.~Bagshaw and I.~Shparlinski.
\newblock Energy bounds, bilinear forms and their applications in function fields.
\newblock {\em Finite Fields Appl.}, 82:102048, 2022.

\bibitem{baker2012}
R.~Baker.
\newblock Kloosterman sums with prime variable.
\newblock {\em Acta Arith.}, 156(4):351--372, 2012.

\bibitem{BanksHarcharrasShparlinski2003}
W.~Banks, A.~Harcharras, and I.~Shparlinski.
\newblock Short {K}loosterman sums for polynomials over finite fields.
\newblock {\em Canad. J. Math.}, 55(2):225--246, 2003.

\bibitem{BhowmickLeLiu2017}
A.~Bhowmick, T.~H. L{\^e}, and Y.~R. Liu.
\newblock A note on character sums in finite fields.
\newblock {\em Finite Fields Appl.}, 46:247--254, 2017.

\bibitem{BourgainGaraev2013}
J.~Bourgain and M.~Z. Garaev.
\newblock {K}loosterman sums in residue rings.
\newblock \url{ https://doi.org/10.48550/arXiv.1309.1124}, 2013.

\bibitem{BourgainGaraev2014}
J.~Bourgain and M.~Z. Garaev.
\newblock Sumsets of reciprocals in prime fields and multilinear {K}loosterman sums.
\newblock {\em Izv. Math.}, 78:656, 2014.

\bibitem{CS2013}
J.~Cilleruelo and I.~Shparlinski.
\newblock Concentration of points on curves in finite fields.
\newblock {\em Monatsh. Math.}, 171(3):315--327, 2013.

\bibitem{fouvrymichel1998}
{\'E}.~Fouvry and P.~Michel.
\newblock Sur certaines sommes d'exponentielles sur les nombres premiers.
\newblock {\em Ann. Sci. Ec. Norm. Sup´er.}, 31(1):93--130, 1998.

\bibitem{fouvryshparlinski2011}
{\'E}.~Fouvry and I.~Shparlinski.
\newblock On a ternary quadratic form over primes.
\newblock {\em Acta Arith.}, 150(3):285--314, 2011.

\bibitem{friedlanderiwaniec}
J.~Friedlander and H.~Iwaniec.
\newblock The {B}run-{T}itchmarsh theorem.
\newblock {\em London Math. Soc. Lecture Note Ser}, 247, 1997.

\bibitem{garaev2010}
M.~Z. Garaev.
\newblock Estimation of {K}loosterman sums with primes and its application.
\newblock {\em Math. Notes}, 88:330--337, 2010.

\bibitem{H2020}
D.~Han.
\newblock A note on character sums in function fields.
\newblock {\em Finite Fields Appl.}, 68:101734, 2020.

\bibitem{Hayes1966}
D.~Hayes.
\newblock The expression of a polynomial as a sum of three irreducibles.
\newblock {\em Acta. Arith.}, 11:461--481, 1966.

\bibitem{irving2014}
A.~Irving.
\newblock Average bounds for {K}loosterman sums over primes.
\newblock {\em Funct. Approx. Comment. Math.}, 51(2):221--235, 2014.

\bibitem{IwaniecKowalski2004}
H.~Iwaniec and E.~Kowalski.
\newblock {\em Analytic number theory}, volume~53.
\newblock American Mathematical Soc., 2004.

\bibitem{karatsuba1995fractional}
A.~Karatsuba.
\newblock Fractional parts of functions of a special form.
\newblock {\em Izv. Math.}, 59:721--740, 1995.

\bibitem{korolevchange2020}
M.~A. Korolev and M.~E. Changa.
\newblock New estimate for {K}loosterman sums with primes.
\newblock {\em Math. Notes}, 108(1-2):87--93, 2020.

\bibitem{luo1999}
W.~Luo.
\newblock Bounds for incomplete hyper-{K}loosterman sums.
\newblock {\em J. Number Theory}, 75(1):41--46, 1999.

\bibitem{Rosen2013}
Michael Rosen.
\newblock {\em Number theory in function fields}, volume 210.
\newblock Springer Science \& Business Media, 2013.

\bibitem{Sawin2023}
W.~Sawin.
\newblock Square-root cancellation for sums of factorization functions over square-free progressions in {F}q[t].
\newblock {\em Acta Math. (to appear)}.

\bibitem{SawinShusterman2}
W.~Sawin and M.~Shusterman.
\newblock M{\"o}bius cancellation on polynomial sequences and the quadratic {B}ateman—{H}orn conjecture over function fields.
\newblock {\em Invent. Math.}, 229(2):751--927, 2022.

\bibitem{sawinshusterman}
W.~Sawin and M.~Shusterman.
\newblock On the {C}howla and twin primes conjectures over {Fq[T]}.
\newblock {\em Ann. Math.}, 196:457--506, 2022.

\bibitem{SZ2018}
I.~Shparlinski and A.~Zumalac{\'a}rregui.
\newblock Sums of inverses in thin sets of finite fields.
\newblock {\em Proc. Amer. Math. Soc.}, 146:1377--1388, 2018.

\bibitem{shparlinski2007}
I.~E. Shparlinski.
\newblock Bounds of incomplete multiple {K}loosterman sums.
\newblock {\em J. Number Theory}, 126(1):68--73, 2007.

\end{thebibliography}

\end{document}